\documentclass[a4paper,12pt]{article}
\usepackage{amsmath,amsthm,amssymb,latexsym,epsfig,graphicx}
\usepackage{enumerate}

\title{{\bf $L^p$ estimates for the maximal singular integral in terms of the singular integral }}
\author{\Large{\Large Anna Bosch-Cam\'{o}s, Joan Mateu and  Joan Orobitg }}
\newtheorem{teorema}{Theorem}

\newtheorem{co}{Corollary}
\newtheorem{lemma}[co]{Lemma}

\theoremstyle{definition}

\newtheorem{remark}{Remark}
\newtheorem*{gracies}{Acknowledgements}

\newcommand{\Rn}{{\mathbb R}^n}
\newcommand{\ep}{\varepsilon}

\newcommand{\BC}{\Rn \setminus \overline B}

\begin{document}

\date{}

\maketitle

\begin{abstract}

This paper continues the study, initiated in the works \cite{MOV} and \cite{MOPV}, 
of the problem of controlling the maximal singular integral $T^{*}f$ by the singular integral $Tf$.
Here $T$ is a smooth homogeneous Calder\'{o}n-Zygmund singular integral operator 
of convolution type. We consider two forms of control, namely, in
the weighted $L^p(\omega)$ norm and via pointwise estimates of $T^{*}f$ by
$M(Tf)$ or $M^2(Tf)$\,, where $M$ is the Hardy-Littlewood maximal
operator and $M^2=M \circ M$ its iteration. 
The novelty with respect to the aforementioned works, lies in the fact that here $p$ is different from $2$
and the $L^p$ space is weighted. 
\end{abstract}

\section{Introduction}

Let $T$ be a smooth homogeneous Calder\'on-Zygmund singular integral operator on $\mathbb R^n$ with kernel 
\begin{equation} \label{nucli}
K(x)=\frac{\Omega(x)}{|x|^n} \, \, \, \,  x\in \mathbb R^n\setminus \{0\},
\end{equation}
where $\Omega$ is a homogeneous function of degree 0 whose restriction 
to the unit sphere $S^{n-1}$ is $C^{\infty}$ and satisfies the cancellation property 
$$
\int_{|x|=1}\Omega(x)d\sigma(x)=0, 
$$
 $\sigma$ being the normalized surface measure in $S^{n-1}$. Thus, $Tf$ is the principal value convolution operator 
\begin{equation} \label{def}
Tf(x)= \text{p.v.} \int f(x-y)K(y)dy\equiv\lim_{\varepsilon \rightarrow 0}T^{\ep}f(x), 
\end{equation}
where $T^{\ep}f$ is the truncated operator at level ${\ep}$ defined by 
$$
T^{\ep}f(x)=\int_{|x-y|>\ep}f(x-y)K(y)dy. 
$$
For $f \in L^p, \, \, 1\le p < \infty$, the limit in \eqref{def} exits for almost all $x$. 
One says that the operator $T$ is even (or odd) if the kernel \eqref{nucli} is even (or odd), that is, if
$\Omega(-x)= \Omega(x)$, $x\in \mathbb R^n\setminus \{0\}$ (or $\Omega(-x)= -\Omega(x)$, $x\in \mathbb R^n\setminus \{0\}$).
Let $T^*$ be the maximal singular integral 
$$
T^*f(x)=\sup_{\ep >0}|T^{\ep}f(x)|, \, \, \, x\in \mathbb R^n.
$$
In this paper we consider the problem of characterizing those smooth Calder\'on-Zygmund  operators for which one can control 
 $T^*f$ by $Tf$ in the weighted $L^p$ norm 
\begin{equation}\label{pnorm}
\|T^*f\|_{L^p(\omega)}\le C\|Tf\|_{L^p(\omega)}, \, \,  \, \, f\in L^p(\omega), \text{ and } \omega \in A_p, 
\end{equation}
where $A_p$ is the Muckenhoupt class of weights (see below for the definition). A stronger way of saying 
that $T^*$ is controlled by $T$ is the pointwise inequality 
\begin{equation}\label{pointwise}
T^*f(x)\le C(M^s(Tf)(x)), \, \, x\in \mathbb R^n, \quad s\in\{ 1,2 \},
\end{equation}
where $M$ denotes the Hardy-Littlewood maximal operator and $M^2=M\circ M$ its iteration. 
For the case $p=2$ and $\omega=1$,  the relationship between \eqref{pnorm} and \eqref{pointwise} has been studied in 
\cite{MOV} for even kernels and in \cite{MOPV} for odd kernels (see also \cite {MV}).  We will 
prove that, for any $1<p<\infty$ and $\omega\in A_p$, the class of operators satisfying \eqref{pnorm}
coincides with the family of operators obtained for $p=2$ and $\omega=1$,
thus giving an affirmative answer to Question 1 of \cite[p. 1480]{MOV}.
Our main result states that for smooth Calder\'on-Zygmund  operators, inequality \eqref{pointwise} (with $s$ 
depending on the parity of the kernel) is equivalent to \eqref{pnorm} and also is equivalent 
to an algebraic condition involving the expansion of $\Omega$ in spherical harmonics. 

Now we need to introduce some notation. The homogeneous function $\Omega$, like any square-integrable function in 
$S^{n-1}$ with zero integral, has an expansion in spherical harmonics of the form

\begin{equation}\label{expansion}
\Omega(x)=\sum_{j=1}^{\infty}P_j(x), \, \, \, \, x\in S^{n-1}, 
\end{equation}
where $P_j$ is a homogeneous harmonic polynomial of degree $j$. For the case of even operators in the above sum we only have the 
even terms $P_{2j}$  and for the odd case we only have the polynomials of odd degree $P_{2j+1}$. In any case, when $\Omega$
is infinitely differentiable on the unit sphere one has that, for each positive integer $M$,
\begin{equation}\label{infinit}
\sum_{j=1}^{\infty}j^M \|P_j\|_{\infty} < \infty,
\end{equation}
where the supremum norm is taken on $S^{n-1}$. When $\Omega$ is of the form
$$
\Omega(x)=\frac{P(x)}{|x|^d}, \quad x \in \Rn \setminus \{0\}\,,
$$
with $P$ a homogeneous harmonic polynomial of degree $d \geq 1$, one says that $T$ is a higher order Riesz transform.  
If the homogeneous polynomial $P$ is not required to be harmonic, but
has still zero integral on the unit sphere, then we call~$T$~a polynomial operator.

Let's recall the definition of Muckenhoupt weights. Let $\omega$ be a non negative locally 
integrable function, and $1<p<\infty$. Then $\omega\in A_p$ if and only if there exits a constant $C$ such that for all cubes
$Q\subset\mathbb R^n$
$$
\left(\frac{1}{|Q|}\int_Q\omega\right) \left(\frac{1}{|Q|}\int_Q\omega^{-\frac{1}{p-1}}\right)^{p-1}\le C.
$$
The important fact worth noting is that Calder\'on-Zygmund  operators and the Hardy-Littlewood maximal operator are bounded  
on $L^p(\omega)$, when $1<p<\infty$ and $\omega$ belongs to $A_p$.
See \cite[Chapter 7]{Du} or \cite[Chapter 9]{Gr2} to get more information on weights.

Now we state our result. We start with the case of even operators.

\begin{teorema}\label{teo1}
 Let $T$ be an even smooth homogeneous
Calder\'on-Zygmund operator with kernel (\ref{nucli}). Then the
following are equivalent:
\begin{enumerate}[(a)]
\item $$T^{*}f(x) \leq CM(Tf)(x), \quad x\in \Rn.$$
\item If  $p\in (1,\infty)$ and $\omega\in A_p$, then 
$$||T^{*}f||_{L^p(\omega)} \leq C||Tf||_{L^p(\omega)}, \quad
\text{for all } f\in L^p(\omega). $$

\item Assume that the expansion \eqref{expansion}
of $\Omega$ in spherical harmonics is
$$
\Omega(x) = \sum_{j=j_0}^\infty P_{2j}(x),\quad P_{2j_0} \neq 0\,.
$$
Then, for each $j$ there exists a homogeneous polynomial
$Q_{2j-2j_0}$ of degree $2j-2j_0$ such that
$P_{2j}=P_{2j_0}Q_{2j-2j_0}$ and
$\sum_{j=j_0}^{\infty}\gamma_{2j}Q_{2j-2j_0}(\xi)\neq 0$, $\xi \in
S^{n-1}$. Here for a positive integer $k$ we have set 
\begin{equation}\label{gamma}
 \gamma_k=i^{-k}\pi^{\frac{n}{2}}\frac{\Gamma(\frac{k}{2})}{\Gamma(\frac{n+k}{2})}.
\end{equation}

\item  
$$||T^{*}f||_{1,\infty}\leq C||Tf||_1, \quad \text{for all }  f\in H^1(\Rn).
$$
\end{enumerate}
\end{teorema}
Recall that $|| g ||_{1,\infty}$ denotes the weak $L^1$ norm of $g$ and $H^1(\Rn)$ is the Hardy space. 
Calder\'on-Zygmund  operators act on $H^1$. (For instance, see \cite[Chapter 6]{Du}, 
\cite[Chapter 7]{Gr2} for more information on the Hardy space).

To get the above result for odd kernels we will replace the Hardy-Littlewood maximal operator in \textit{(a)} by its
iteration.
\begin{teorema}\label{teo2}
 Let $T$ be an odd smooth homogeneous Calder\'on-Zygmund operator with kernel (\ref{nucli}). 
Then the following are equivalent:
\begin{enumerate}[(a)]
\item $$T^{*}f(x) \leq CM^2(Tf)(x), \quad x\in \Rn.$$
\item If $p\in (1,\infty)$ and $\omega\in A_p$ then 
$$||T^{*}f||_{L^p(\omega)} \leq C||Tf||_{L^p(\omega)}, \quad
\text{for all } f\in L^p(\omega).$$ 
\item Assume that the expansion \eqref{expansion}
of $\Omega$ in spherical harmonics is
$$
\Omega(x) = \sum_{j=j_0}^\infty P_{2j+1}(x),\quad P_{2j_0+1} \neq 0\,.
$$
Then, for each $j$ there exists a homogeneous polynomial
$Q_{2j-2j_0}$ of degree $2j-2j_0$ such that
$P_{2j+1}=P_{2j_0+1}Q_{2j-2j_0}$ and
$\sum_{j=j_0}^{\infty}\gamma_{2j+1}Q_{2j-2j_0}(\xi)\neq 0$, $\xi \in
S^{n-1}$,
with $\gamma_{2j+1}$ as in  \eqref{gamma}.
\end{enumerate}
\end{teorema}

Clearly, both in Theorem \ref{teo1}  as in Theorem \ref{teo2}, the condition \textit{(a)} implies \textit{(b)}
is a consequence of the boundedness of the Hardy-Littlewood maximal operator on weighted $L^p$ spaces. 
The proof of \textit{(c)} implies \textit{(a)} in Theorem \ref{teo1} is proved in \cite{MOV} and the same implication
in Theorem \ref{teo2} is proved in  \cite{MOPV}.
So the only task to be done is to show that \textit{(b)} implies \textit{(c)} in both theorems 
(and $(d) \Rightarrow (c)$ in Theorem \ref{teo1}).
One of the crucial points in the proof of the implication $(b) \Rightarrow (c)$ for the case $p=2$ and $\omega=1$ in  
\cite{MOV} and \cite{MOPV} is to use Plancherel Theorem 
to get a pointwise inequality to work with it. For $p\neq2$ we will get the corresponding pointwise inequality using 
properties of the Fourier transform of the kernels as $L^p$ multipliers. 

In Section \ref{sec2} we introduce $L^p$ Fourier multipliers and some tools to control their norm (see Lemma \ref{L1}).
Section \ref{sec3} is devoted to the proof of $(b) \Rightarrow (c)$, for polynomial operators. The general case is discussed in Section
\ref{sec4}.

As usual, the letter $C$ will denote a constant, which may be different at each occurrence and which is independent 
of the relevant variables under consideration.

\section{Multipliers}\label{sec2}

Recall that, given $1\le p<\infty$, one denotes by $\mathcal M_p(\mathbb R^n )$ the space of all bounded
functions $m$ on $\mathbb R^n$ such that the operator
$$
T_m ( f ) = ( \hat f \; m)^{\vee} ,\quad  f\in\mathcal{S} ,
$$
is bounded on $L^p(\mathbb R^n )$ (or is initially defined in a dense subspace of $L^p(\mathbb R^n )$ and has
a bounded extension on the whole space). As usual, $\mathcal{S}$ denotes the space of Schwartz functions, $\hat f$ is the
Fourier transform of $f$ and $f^{\vee}$ the inverse Fourier transform.
The norm of $m$ in $\mathcal M_p(\mathbb R^n )$ is defined as
the norm of the bounded linear operator $T_m\, : L^p(\mathbb R^n )\to L^p(\mathbb R^n )$.
Elements of the space $\mathcal M_p(\mathbb R^n )$ are called  $L^p$ (Fourier) multipliers. Similarly, we speak of 
$L^p(\omega)$ multipliers.
It is well known that 
$\mathcal M_2$, the set of all $L^2$ multipliers, is $L^{\infty}$ and that $\mathcal M_1(\mathbb R^n )$ is the set of
Fourier transforms of finite Borel measures on $\mathbb R^n$. The basic theory on multipliers 
may be found for example in the monographs \cite{Du},\cite{Gr1}.

Let $0\leq \phi \leq 1$ be an smooth function such that $\phi(\xi)=1$ if $|\xi |\leq \frac{1}{2}$, and $\phi(\xi)=0$ 
if $|\xi|\geq 1$. Given $\xi_0\in\mathbb{R}^n $, we define $\phi_{\delta}(\xi)=\phi(\frac{\xi- \xi_0}{\delta})$. 
Consider $m\in L^{\infty}$ such that
$m$ is continuous in some neighbourhood of $\xi_0$ with $m(\xi_0) =0$. It is clear, by Plancherel Theorem, that the norm
of $m \phi_{\delta}$ in $\mathcal M_2$ approaches zero when $\delta\to 0$. We ask if the same result holds when $m$ is
an $L^p$ multiplier. Adding some regularity to $m$ we get a positive answer.

\begin{lemma}\label{L1}
Let $\xi_0\in \mathbb{R}^n$, $0< \delta\le \delta_0$ and  $m\in\mathcal M_p\cap \mathcal{C}^n(B(\xi_0,\delta_0))$ with $m(\xi_0)=0$.
 Let $\phi\in \mathcal{C}^{\infty}(\Rn)$, $0\leq \phi \leq 1$ such that $\phi(\xi)=1$ if $|\xi |\leq \frac{1}{2}$, and $\phi(\xi)=0$ 
if $|\xi|\geq 1$. Set $\phi_{\delta}(\xi)=\phi(\frac{\xi- \xi_0}{\delta})$ and
let $T_{m \phi_{\delta}}$ be the operator with multiplier $m \phi_{\delta}$.
\begin{enumerate}
 \item If $\omega\in A_p$, $1<p<\infty$, then $||T_{m \phi_{\delta}}||_{L^p(\omega)\rightarrow L^p(\omega)}
 \longrightarrow 0$, when $\delta\rightarrow 0$.
 \item $||T_{m \phi_{\delta}}||_{L^1 \rightarrow L^{1,\infty}}\longrightarrow 0$,
 when $\delta\rightarrow 0$.
 \item $||T_{m \phi_{\delta}}||_{H^1\rightarrow L^1}\longrightarrow 0$,  when $\delta\rightarrow 0$.
  
\end{enumerate}
\end{lemma}

To prove Lemma \ref{L1}, we use the next theorem due to Kurtz and Wheeden. 
Following \cite{KW}, we say that a function $m$ belongs to the class $M(s,l)$ if
\begin{equation}\label{msl}
 m_{s,l}:=\sup_{R>0}\left(R^{s|\alpha|-n}\int_{R<|x|<2R}|D^{\alpha}m(x)|^s dx\right)^{1/s}<+\infty, 
\mbox{    for all }|\alpha|\leq l,
\end{equation}
where $s$ is a real number greater or equal to $1$, $l$ a
positive integer and $\alpha=(\alpha_1,\ldots,\alpha_n)$ a
multiindex of nonnegative integers.

\begin{teorema}{\cite[p. 344]{KW}}\label{Kurtz}
Let $1<s\leq 2$ and $m\in M(s,n)$.
\begin{enumerate}
\item If $1<p<\infty$ and $\omega \in A_{p}$, then there exists a constant $C$, independent of $f$,
 such that
$$||T_m f||_{L^p(\omega)}\leq C||f||_{L^p(\omega)}.$$
\item There exists a constant $C$, independent of $f$ and $\lambda$, such that
$$ |\{x\in \mathbb R^n : |T_mf(x) |> \lambda\} |\leq \frac{C}{\lambda}\|f\|_{L^1}, \quad \lambda>0.$$
\item  There exists a constant $C$, independent of $f$, such that
$$||T_mf||_{L^1}\leq C||f||_{H^1}.$$
\end{enumerate}
\end{teorema}

Analyzing the proof we check that, in all cases, the constant $C$,
which appears in the statements \textit{1}, \textit{2} and  \textit{3} of the previous Theorem, 
depends linearly on the constant $m_{s,n}$ defined at \eqref{msl}.
We also remark that when $\omega=1$ the
proof can be adapted to the case $H^1\rightarrow L^1$, so we get
statement \textit{3} which is not explicitly written in \cite{KW}.\\

\begin{proof}[Proof of Lemma \ref{L1}]

Using Theorem \ref{Kurtz} we only need to prove that the multiplier $m\phi_{\delta}$ is in 
$M(s,n)$ for some $1<s\le 2,$ and the constant $m_{s,n}$ tends to $0$ if $\delta$ tends to $0$.

Assume that $\xi_0 \ne 0$ and that $\delta<\delta_0$ is small enough.
For $|\alpha|\leq n$, using Leibniz rule one has  

\begin{align*}
&\quad\sup_{R>0}\left(R^{s|\alpha|-n}\int_{R<|\xi|<2R}|D^{\alpha}(m
\phi_{\delta})(\xi)|^s d\xi\right)^{1/s}\\
&= \sup_{R>0}\left(R^{s|\alpha|-n}\int_{\{R<|\xi|<2R\}\cap B(\xi_0, \delta)}|D^{\alpha}(m \phi_{\delta})(\xi)|^s d\xi\right)^{1/s}\\
 &\leq C|\xi_0|^{|\alpha|-\frac{n}{s}}\left(\int_{B(\xi_0, \delta)}|D^{\alpha}(m \phi_{\delta})(\xi)|^s d\xi\right)^{1/s}\\
&\leq C |\xi_0|^{|\alpha|-\frac{n}{s}}
\left(\sum_{\beta_{i}\le \alpha_{i}, \, 1\le i\le n}\binom{\alpha_1}{\beta_1}\binom{\alpha_2}{\beta_2}\cdots\binom{\alpha_n}{\beta_n}
\int_{B(\xi_0,
\delta)}|D^{\alpha-\beta}(m)(\xi)D^{\beta}(\phi_{\delta})(\xi)|^s
d\xi\right)^{1/s}.
\end{align*}

Now we will get a bound for each term in the above sum. In order to get it, we consider
different cases. In all the cases we will use that for any multiindex $\alpha$ we have  $|D^{\alpha}\phi_{\delta}(\xi)|\lesssim
\frac{1}{\delta^{|\alpha|}}$ and  that the modulus of continuity of $m,$ denoted by $\omega(m,\xi_0, \delta)$, satisfies
$\omega(m,\xi_0, \delta)\le C\delta$.

Case 1. $|\alpha|=n$.

For $\beta=\alpha$ one has that 

\begin{align*}
\int_{B(\xi_0, \delta)}|D^{\alpha-\beta}(m)(\xi)D^{\beta}(\phi_{\delta})(\xi)|^s d\xi &=
 \int_{B(\xi_0, \delta)}|m(\xi)|^{s}|D^{\alpha}(\phi_{\delta})(\xi)|^s d\xi \\
 &\leq C \frac{1}{\delta^{ns}}|\omega(m,\xi_0, \delta)|^s \delta^n\\
&\le C\delta^{s+n-ns} 
\end{align*}
and this term tends to $0$ as $\delta$ tends to $0$ taking $1<s<\frac{n}{n-1}$. 
For the remaining terms, that is $\alpha\neq\beta$, we have 

\begin{align*}
\int_{B(\xi_0, \delta)}|D^{\alpha-\beta}(m)(\xi)D^{\beta}(\phi_{\delta})(\xi)|^s d\xi
&\leq C \frac{1}{\delta^{|\beta|s}}\delta^n \\
&=C \delta^{n-s|\beta|}\\
&\le C\delta^{s+n-ns},
\end{align*}
where the derivatives of $m$ are bounded by a constant, and the last inequality holds when $\delta$ is small enough.
So, if $1<s<\frac{n}{n-1}$, this term goes to $0$ as $\delta$ goes to $0.$

Case 2. $|\alpha|=k<n$.

For $|\beta|=|\alpha|$, using the boundedness of the modulus of continuity of $m$ we have

\begin{align*}
\int_{B(\xi_0, \delta)}|D^{\alpha-\beta}(m)(\xi)D^{\beta}(\phi_{\delta})(\xi)|^s d\xi
&= \int_{B(\xi_0, \delta)}|m(\xi)|^{s}|D^{\alpha}(\phi_{\delta})(\xi)|^s d\xi \\
&\leq C\frac{1}{\delta^{ks}}|\omega(m,\xi_0, \delta)|^s \delta^n \\
&= C\delta^{s+n-ks}\\
&\le C\delta^{s+n-ns}\
\end{align*}

and this term, again, goes to $0$ as $\delta$ goes to $0$, whenever $1<s<\frac{n}{n-1}$. 

Finally, if  $|\beta|<|\alpha|$, one gets the same bound

\begin{align*}
\int_{B(\xi_0, \delta)}|D^{\alpha-\beta}(m)(\xi)D^{\beta}(\phi_{\delta})(\xi)|^s d\xi
&\leq C \frac{1}{\delta^{|\beta|s}}\delta^n \\
&= C \delta^{n-s|\beta|}\\
&\le C \delta^{s+n-ns}.
\end{align*}

When $\xi_0 =0$ one has
\begin{align*}
&\sup_{R>0}\left(R^{s|\alpha|-n}\int_{R<|\xi|<2R}|D^{\alpha}(m
\phi_{\delta})(\xi)|^s d\xi\right)^{1/s}\\
= &\sup_{\delta\ge R>0}\left(R^{s|\alpha|-n}\int_{R<|\xi|<2R}|D^{\alpha}(m
\phi_{\delta})(\xi)|^s d\xi\right)^{1/s}.
\end{align*}
Observe that for $|\alpha|>0$, $D^{\alpha}\phi_{\delta}$ lives on $\{\delta/2 \le |\xi|\le \delta\}$. 
Then, similar calculations complete the proof.
\end{proof}

To prove the first case of Lemma 1 there is another argument due to J. Duoandikoetxea. 
We thank him for providing us the following lemma. In fact, it is only necessary to assume that the multiplier
$m$ is continuous.

\begin{lemma}\label{L2}
 Let $\xi_0\in \mathbb{R}^n$, $0< \delta\le \delta_0$, $1<q<2$ and  $m\in\mathcal M_q\cap \mathcal{C}(B(\xi_0,\delta_0))$ 
with $m(\xi_0)=0$.
Set $\phi_{\delta}(\xi)$ as above and let $T_{m \phi_{\delta}}$ be the operator with multiplier $m \phi_{\delta}$.
\begin{enumerate}[(a)]
\item For any $p\in (q,2)$ we have  
$$
\|T_{m \phi_{\delta}}\|_{L^p \rightarrow L^p }
 \longrightarrow 0,\quad \text{when } \; \delta\rightarrow 0 .
 $$
\item Let $\omega\in A_p$ with $p\in (q,2)$ and let $s>1$ such that $\omega^s\in A_p$. 
If $m$ is an $L^p(\omega^s)$ multiplier,  then 
$$
\|T_{m \phi_{\delta}}\|_{L^p(\omega) \rightarrow L^p(\omega) }
 \longrightarrow 0,\quad \text{when } \; \delta\rightarrow 0 .
 $$
\end{enumerate}
\end{lemma}
\begin{remark}
Clearly, a similar result holds when $2<p<q$.
\end{remark}

\begin{proof} We first observe that $||T_{m\phi_{\delta}}||_{L^2\rightarrow
L^2}=\|m\phi_{\delta}\|_{\infty}=\varepsilon(\delta)$ and $\varepsilon(\delta)\to 0$ as $\delta\to 0$ since $m$ 
is continuous in $\xi_0$.  On the other hand, $\| m \phi_{\delta}\|_{\mathcal M_q} \le \| \phi_{\delta}^{\vee}  \|_{L^1} 
\| m \|_{\mathcal M_q} = C \| m \|_{\mathcal M_q}$, where $C$ is a constant independent of $\delta$. That is, for all $\delta>0$
$$
\|T_{m\phi_{\delta}}f\|_{q}\leq M\|f\|_{q}
$$

Then, applying the Riesz-Thorin theorem (e.g. \cite[p. 34]{Gr1}), for any $p\in (q,2)$ 
($\frac{1}{p}=\frac{1-\theta}{2}+\frac{\theta}{q}$) we have
\begin{equation}\label{una}
\|T_{m\phi_{\delta}}f\|_p\leq
M^{1-\theta}\varepsilon(\delta)^{\theta}\|f\|_p = \varepsilon_1(\delta)\|f\|_p,
\end{equation}
where $\varepsilon_1(\delta)\to 0$ as $\delta\to 0$
and (a) is proved. For proving (b), since $\omega^s\in A_p$ and $\phi_{\delta}$ is a cutoff smooth function, note that
\begin{equation}\label{dues}
\|T_{m\phi_{\delta}}f\|_{L^p(\omega^s)}\leq C \|f\|_{L^p(\omega^s)},
\end{equation}
where one can check that $C$ is a constant independent of $\delta$. Finally, from \eqref{una} and \eqref{dues}, applying
the interpolation theorem with change of measure of Stein-Weiss (e.g. \cite[p. 115]{BeL}), we get
$$
\|T_{m\phi_{\delta}}f\|_{L^p(\omega)}\leq C^{1/s} \varepsilon_1(\delta)^{1-1/s} \|f\|_{L^p(\omega)}
$$
as desired.
\end{proof}

\section{The polynomial case}\label{sec3}

As we remarked in the Introduction, to have a complete proof of theorems \ref{teo1} and \ref{teo2} only remains to prove
that \textit{(b)} implies \textit{(c)}  (and \textit{(d)} implies \textit{(c)} in Theorem \ref{teo1}).
Our procedure to get the above implications follows essentially the arguments used in \cite{MOV} and \cite{MOPV}. 
The main difficulty to overcome is that
for $p\neq 2$, we cannot apply Plancherel Theorem and we replace it by a Fourier multiplier argument.\\

We begin with the proof of \textit{(b)} implies \textit{(c)} in Theorem \ref{teo1} for the case $\omega=1$. 
Then we show how to adapt this proof to the case with weights, to the case of odd operators and to the case of weak $L^1$. 
Thus, we assume that $T$ is an even polynomial operator with kernel

$$K(x)=\frac{\Omega(x)}{|x|^n}=\frac{P_2(x)}{|x|^{2+n}}+\frac{P_4(x)}{|x|^{4+n}}+\ldots+\frac{P_{2N}(x)}{|x|^{2N+n}}, 
\quad x\neq 0,$$
where $P_{2j}$ is a homogeneous harmonic polynomial of degree
$2j$. Each term  has the multiplier (see \cite[p. 73]{St})
$$
 \left( \text{p.v.}\frac{P_{2j}(x)}{|x|^{2j+n}}\right) ^{\wedge}(\xi) = \gamma_{2j}\frac{P_{2j}(\xi)}{|\xi|^{2j+n}},
$$
Then, 
$$
\widehat{\text{p.v.} K}(\xi) = \frac{Q(\xi)}{|\xi|^{2N}},\qquad \xi\neq 0,
$$
where $Q$ is the homogeneous polynomial of degree $2N$ defined
by
$$Q(x)=\gamma_2 P_2(x)|x|^{2N-2}+\ldots+\gamma_{2j} P_{2j}(x)|x|^{2n-2j}+\ldots+\gamma_{2N}P_{2N}(x).
$$
We want to obtain a convenient expression for the function $K(x)\chi_{\Rn\setminus\overline{B}}$, the kernel $K$ off the unit ball $B$
(see \eqref{expressio}).
To find it, we need a simple technical lemma which we state without proof.

\begin{lemma}{\cite[p. 1435]{MOV}}\label{derivarbandaibanda}
Assume that $\varphi$ is a radial function of the form
$$\varphi(x)=\varphi_1(|x|)\chi_B(x)+\varphi_2(|x|)\chi_{\Rn\setminus\overline{B}}(x),$$
where $\varphi_1$ is continuously differentiable on $[0,1)$ and $\varphi_2$ on $(1,\infty)$. 
Let $L$ be a second order linear differential operator with constant coefficients. Then the distribution $L\varphi$ satisfies
$$L\varphi=L\varphi(x)\chi_B(x)+L\varphi(x)\chi_{\Rn\setminus\overline{B}}(x),$$
provided $\varphi_1$, $\varphi'_1$, $\varphi_2$ and $\varphi'_2$ extend continuously to the point $1$ and the two conditions
$$\varphi_1(1)=\varphi_2(1),\quad \varphi'_1(1)=\varphi'_2(1)$$
are satisfied.
\end{lemma}

Consider the differential operator $Q(\partial)$ defined by the polynomial $Q(x)$ above and let $E$ be the standard 
fundamental solution of the $N$-th power $\Delta^N$ of the Laplacian. 
Then $Q(\partial)E= \text{p.v.}K(x)$, which may be verified by taking the Fourier transform of both sides.
The concrete expression of $E(x)= |x|^{2N-n}(a(n,N)+b(n,N)\log |x|^2)$ (e.g. \cite[p. 1464]{MOV}) is not important now, 
just note that it is a radial function.
Consider the function
$$\varphi(x)=E(x)\chi_{\Rn\setminus \overline{B}}(x)+(A_0+A_1|x|^2+\cdots+A_{2N-1}|x|^{4N-2})\chi_B(x),$$
where $B$ is the open ball of radius 1 centered at origin and the constants $A_0, A_1, \ldots,$ $ A_{2N-1}$ are chosen as
follows. Since $\varphi(x)$ is radial, the same is true for
$\Delta^j\varphi$ if $j$ is a positive integer. Thus, in order to
apply $N$ times Lemma \ref{derivarbandaibanda}, one needs
$2N$ conditions, which (uniquely) determine $A_0, A_1, \ldots,
A_{2N-1}$. Therefore, for some constants $\alpha_1, \alpha_2,
\ldots, \alpha_{N-1}$,
\begin{equation}\label{b}
\Delta^N \varphi=(\alpha_0+\alpha_1|x|^2+\cdots\alpha_{N-1}|x|^{2(N-1)})\chi_B(x)=b(x),
\end{equation}
where the last identity is the definition of $b$. Let's remark that $b$ is a bounded function supported in the unit ball and
it only depends on $N$ and not on the kernel $K$. Since
$$\varphi=E\ast \Delta^N\varphi,$$
taking derivatives of both sides we obtain
$$Q(\partial)\varphi=Q(\partial)E\ast \Delta^N \varphi= \text{p.v.} K(x) \ast b=T(b).$$
On the other hand, applying Lemma \ref{derivarbandaibanda},
$$Q(\partial)\varphi=K(x)\chi_{\Rn\setminus \overline{B}}(x)+Q(\partial)(A_0+A_1|x|^2+\cdots+A_{2N-1}|x|^{4N-2})(x)\chi_B(x).$$
We write
$$S(x):=-Q(\partial)(A_0+A_1|x|^2+\ldots+A_{2N-1}|x|^{4N-2})(x),$$
and we get
\begin{equation}\label{expressio}
 K(x)\chi_{\Rn\setminus\overline{B}}(x)=T(b)(x)+S(x)\chi_{B}(x).
\end{equation}
Let's remark that $S$ will be null when $Q$ is a harmonic polynomial (see \cite[p. 1437]{MOV}).
Consequently
$$
T^1 f = T(b)*f + S\chi_{B} * f.
$$
Our assumption is the $L^p$ estimate between $T^*$ and $T$. Since the truncated operator $T^1$ at level 1 is obviously
dominated by $T^*$, we have
\begin{equation}\label{cadena}
\begin{split}
\|  S\chi_{B} * f \|_p &\le \| T^1 f \|_p + \| Tb*f \|_p\\
& \le \| T^* f \|_p + \| b*Tf \|_p \\ 
&\le C \| T f \|_p + \| b\|_1 \|Tf \|_p\\
&= C \|Tf\|_p,
\end{split}
\end{equation}
that is, for any $f\in L^p$
\begin{equation}\label{hipo}
\|S\chi_B \ast f\|_p \leq C\| \text{p.v.} K \ast f\|_p.
\end{equation}
If $p=2$, we can use Plancherel and this $L^2$ inequality translates into a pointwise inequality between the Fourier multipliers:
\begin{equation}\label{desmult}
|\widehat{S\chi_B}(\xi)|\leq C|\widehat{\text{p.v.} K}(\xi)| = \frac{Q(\xi)}{|\xi|^{2N}}     ,\quad  \xi\neq 0.
\end{equation}

If $p\neq 2$ we must resort to Fourier multipliers to get \eqref{desmult}.
We observe that the multipliers we are dealing with,
$\widehat{S\chi_B}$ and $\widehat{\text{p.v.}K}$,
are in $\mathcal{C}^{\infty}\setminus \{0\}$ and in $\mathcal{M}_p$.
Let $\xi_0 \neq 0$, we write 
\begin{align*}
 &\widehat{S\chi_B}(\xi) =\widehat{S\chi_B}(\xi)(\xi_0)+E_1(\xi) \quad \text{with }\quad E_1(\xi)= 
 \widehat{S\chi_B}(\xi) -\widehat{S\chi_B}(\xi_0)\\
 &\widehat{\text{p.v.}K}(\xi) =\widehat{\text{p.v.}K}(\xi_0)+E_2(\xi) \quad \text{with }\quad E_2(\xi)=\widehat{\text{p.v.}K}(\xi)-
 \widehat{\text{p.v.}K}(\xi_0)
\end{align*}
and so
\begin{equation}\label{mespetit}
\|\text{p.v.}K \ast f\|_p \leq
|\widehat{\text{p.v.}K}(\xi_0)| \, \|f\|_p + \|T_{E_2}f\|_p,
\end{equation}
\begin{equation}\label{mesgran}
\|S\chi_B \ast f\|_p \geq
|\widehat{S\chi_B}(\xi_0)| \, \|f\|_p -\|T_{E_1}f\|_p,
\end{equation}
where $T_{E_i}$ denotes the operator with multiplier $E_i$ ($i=1,2$). Using \eqref{mesgran}, \eqref{hipo}
and \eqref{mespetit} consecutively, we get
\begin{align*}
 |\widehat{S\chi_B}(\xi_0)|\, \|f\|_p - \|T_{E_1}f\|_p & \leq \|S\chi_B \ast f\|_p\\
                               & \leq C \|\text{p.v.}K \ast f\|_p \\
                               &\leq C (|\widehat{\text{p.v.}K}(\xi_0)|\, \|f\|_p + \|T_{E_2}f\|_p)
\end{align*}
and therefore
\begin{equation} \label{controlinicial} 
|\widehat{S\chi_B}(\xi_0)|\leq C \left(
|\widehat{\text{p.v.}K}(\xi_0)|+\frac{||T_{E_2}f||_p}{||f||_p}+\frac{||T_{E_1}f||_p}{||f||_p}\right),
\quad  \xi_0 \neq 0.
\end{equation}

Now, choosing appropriate functions in \eqref{controlinicial} we obtain the pointwise inequality.
Let $\phi_{\delta}(\xi)=\phi(\frac{\xi- \xi_0}{\delta})$ as in Lemma \ref{L1} and define $g_{\delta}
\in \mathcal{S}(\mathbb R^n)$ by $\widehat{g_{\delta}}(\xi) = \phi_{\delta}(\xi)$.
Then $T_{E_j}g_{\delta}= T_{E_j}  (g_{2\delta} * g_{\delta}) =T_{E_j  \phi_{2\delta}}(g_{\delta})$, because $\phi_{2\delta}=1$ 
on the support of $\phi_{\delta}$. Changing $f$ by $g_{\delta}$ in \eqref{controlinicial} we have

\begin{align*}
|\widehat{S\chi_B}(\xi_0)|& \leq C \left(
|\widehat{\text{p.v.}K}(\xi_0)|+\frac{\|T_{E_2 
\phi_{2\delta}}g_{\delta}\|_p}{||g_{\delta}\|_p}+\frac{\|T_{E_1
\phi_{2\delta}}g_{\delta}\|_p}{\|g_{\delta}\|_p} \right)\\
& \leq C \left(  |\widehat{\text{p.v.}K}(\xi_0)|+\|T_{E_2 \phi_{2\delta}}\|_{L^p\rightarrow
L^p}+\|T_{E_1 \phi_{2\delta}}\|_{L^p\rightarrow L^p}\right) .
\end{align*}

Applying Lemma \ref{L1} to the multipliers $E_j$ we prove that the two
last terms tend to zero as $\delta$ tends to zero. So, for $\omega=1$, we get 
\eqref{desmult} and from here we would follow the arguments in \cite[p. 1457]{MOV}.\\

For the weighted  case  we must be careful with the inequalities in \eqref{cadena}. 
In general, the inequality $\|f* F \|_{L^p (\omega)} \le C\| f\|_1 \| F \|_{L^p (\omega)}$ is not satisfied. That is, 
we can not control $\|b* Tf \|_{L^p (\omega)}$ by a constant times $\| b\|_1 \| Tf \|_{L^p (\omega)}$.
However, in the even case $b$ is a bounded function supported in the unit ball and so 
$$
|(b*Tf)(x)|= \left| \int_{|x-y|<1} b(x-y) Tf(y)\, dy\right|
\le C M(Tf)(x).
$$

Moreover 
$$
\| b*Tf\|_{L^p (\omega)}\le C \| Tf\|_{L^p (\omega)},
$$
because $\omega\in A_p$. So, $\|S\chi_B \ast f\|_{L^p(\omega)} \leq C\|\text{p.v.}K \ast f\|_{L^p(\omega)}$ 
and proceeding as above, we would get \eqref{desmult}.\\

 The proof of \textit{(b)}  implies \textit{(c)} in Theorem \ref{teo2}
 can be handled in much the same way. 
The only significant difference, because now the polynomial 
is odd, lies on the function $b$ in \eqref{expressio}, which is not supported in the unit ball but it is a $\operatorname{BMO}$
function satisfying the decay $|b(x)| \le C |x|^{-n-1}$ if $|x| > 2$ (see \cite[section 4]{MOPV}).
In any case, $b\in L^1$ and the set of inequalities \eqref{cadena} remains valid for the case $\omega=1.$\\
On the other hand, for any $\omega$ in the Muckenhoupt class we write, arguing as in \cite[p. 3675]{MOPV},
\begin{equation*}
\begin{split}
 |(b*Tf)(x)|=  & \left| \int_{|x-y|<2} (b(x-y) - b_{B(0,2)})Tf(y)\, dy\right| + \\
 & + | b_{B(0,2)}| \int_{|x-y|<2} |Tf(y)|\, dy +  \int_{|x-y|>2} |b(x-y)|\,  |Tf(y)|\, dy\\
 & = I + II + III,
\end{split}
\end{equation*}
where $b_{B(0,2)}= |B(0,2)|^{-1}\int_{B(0,2)} b\;$.
To estimate the local term $I$ we use the generalized H\"older's inequality and the pointwise equivalence
$M_{L(\log L)}f(x)\simeq M^2f(x)$ (\cite{P}) to get
$$
|I| \le C\| b \|_{BMO} \| Tf\|_{L(\log L), B(x,2)}\le C M^2(Tf)(x).
$$
Notice that $b_{B(0,2)}$ is a dimensional constant. Hence
$$
|II |\le C M(Tf)(x).
$$
Finally, from the decay of $b$ we obtain
$$
| III| \le C \int_{|x-y|>2} \frac{|Tf(y)|}{|x-y|^{n+1}}\, dy\le C M(Tf)(x),
$$
by using a standard argument which consists in estimating the integral
on the annuli $\{2^k \le |x-y| < 2^{k+1}  \}$. Therefore
\begin{equation}\label{iterada}
 |(b*Tf)(x)|\le C M^2(Tf)(x).
\end{equation}

So,  we obtain 
$$
\| b*Tf\|_{L^p (\omega)}\le C \| Tf\|_{L^p (\omega)},
$$
because $\omega\in A_p$. Then, $\|S\chi_B \ast f\|_{L^p(\omega)} \leq C\|\text{p.v.}K \ast f\|_{L^p(\omega)}$ 
and we get \eqref{desmult}.\\

It remains to prove that \textit{(d)} implies \textit{(c)} in Theorem \ref{teo1}. To get this implication we need to precise
some properties of the functions $g_{\delta}$ that we explain below. First of all, note that
$g_{\delta}(x)= e^{i x \xi_0} \delta^{n} g(\delta x)$ where $\hat g =\phi$. So it is clear that the norms $\| g_{\delta}\|_1 =
\| g\|_1$ and  $\| g_{\delta}\|_{1,\infty} = \| g\|_{1,\infty}$ do not depend on the parameter $\delta>0$. When $\delta<|\xi_0|$,
since $\int g_{\delta}(x) dx = \phi_{\delta}(0)=0$ and $g_{\delta}
\in \mathcal{S}(\mathbb R^n)$, we have that $g_{\delta}\in H^1$. But, some computations are required to check 
that $\|g_{\delta}\|_{H^1}\le C$ with 
constant $C$ independent of $\delta$.
\begin{lemma}\label{aux}
 When $0<\delta<|\xi_0|$, $\|g_{\delta}\|_{H^1}\le C$ with  constant $C$ independent of $\delta$.
\end{lemma}
\begin{proof}
 We have $g_{\delta}(x)= e^{i x \xi_0} \delta^{n} g(\delta x)$ with $g\in \mathcal{S}(\mathbb R^n)$ and $\int g_{\delta}=0$.
 Set $F^{\delta}_0(x) =\chi_{B(0,\delta^{-1})}(x)$ and, for $j\ge 1$, $F^{\delta}_j(x) =\chi_{B(0,2^j\delta^{-1})}(x)
 - \chi_{B(0,2^{j-1}\delta^{-1})}(x)$. Note that $\sum_{j=0}^{\infty}F^{\delta}_j(x)\equiv 1$. Consider the atomic decomposition of
 $g_{\delta}$
 \begin{align*}
   g_{\delta}(x)& =  \sum_{j=0}^{\infty} (g_{\delta}(x)- c^{\delta}_j )  F^{\delta}_j(x)
 + \sum_{j=0}^{\infty} [ (c^{\delta}_j + d^{\delta}_j)  F^{\delta}_j(x) - d^{\delta}_{j+1}  F^{\delta}_{j+1}(x) ]\\
                & := \sum_{j=0}^{\infty} a^{\delta}_j (x) \; + \;  \sum_{j=0}^{\infty}A^{\delta}_j(x),
 \end{align*}
 where $c^{\delta}_j= \dfrac{\int g_{\delta}F^{\delta}_j}{\int F^{\delta}_j}$, $d^{\delta}_0=0$ and 
 $d^{\delta}_{j+1}=  \dfrac{\int g_{\delta}(F^{\delta}_0 + \cdots + F^{\delta}_j)}{\int F^{\delta}_{j+1}}$,
 so that $\int a^{\delta}_j (x)dx= \int A^{\delta}_j (x)dx= 0$. Note that $a^{\delta}_j$ is supported in the ball $B(0,2^j\delta^{-1})$
 and $A^{\delta}_j$ is supported in $B(0,2^{j+1}\delta^{-1})$.
 
 Since $g\in \mathcal{S}(\mathbb R^n)$ we have $(1+ |z|^{n+1})|g(z)|\le C$. Then
 $$
 |g_{\delta}(x) F^{\delta}_j(x)| = \delta^{n} |g(\delta x)| F^{\delta}_j(x)\le \delta^n\sup_{|z|\sim 2^j}|g(z)|
 \le C \left(\frac{\delta}{2^j}\right)^n 2^{-j} = \frac{C 2^{-j}}{|B(0,2^j\delta^{-1})|}
 $$
 and therefore
 $$
 |c^{\delta}_j|= \left|\frac{\int g_{\delta}F^{\delta}_j}{\int F^{\delta}_j}\right|\le \frac{C 2^{-j}}{|B(0,2^j\delta^{-1})|}.
 $$
 
 On the other hand,  $\int g_{\delta}(F^{\delta}_0 + \cdots + F^{\delta}_j)= 
 \int_{|x|\ge 2^{j}\delta^{-1}} g_{\delta}(x) dx$, because $\int g_{\delta}=0$, and so
 $$
 d^{\delta}_{j+1}= \frac{\int_{|x|\ge 2^{j}\delta^{-1}} g_{\delta}(x) dx}{\int F^{\delta}_{j+1} } \le 
 \frac{\int_{|z|\ge 2^j}|g(z)| dz}{|B(0,2^{j+1}\delta^{-1})|} \le \frac{C 2^{-j}}{|B(0,2^{j+1}\delta^{-1})|}.
 $$
 
 Consequently
 $$
 \| a^{\delta}_j\|_{H^1}\le \frac{C}{2^j} \qquad\text{and}\qquad \| A^{\delta}_j\|_{H^1}\le \frac{C}{2^j}.
 $$
 
 Therefore, for all $\delta\in (0, |\xi_0|)$, $\| g_{\delta}\|_{H^1}\le C$ as we claimed. 
 
\end{proof}

Finally, for functions $f$ in $H^1$, and again using \eqref{expressio}, we have
\begin{equation*}
\begin{split}
\|  S\chi_{B} * f \|_{1,\infty} &\le 2( \| T^1 f \|_{1,\infty} + \| Tb*f \|_{1,\infty} )\\
& \le C( \| T^* f \|_{1,\infty} + \| b*Tf \|_1 )\\ 
&\le C \| T f \|_1 + \| b\|_1 \|Tf \|_1)\\
&= C \|Tf\|_1 = C \|\text{p.v.} K*f\| _1.
\end{split}
\end{equation*}
Taking $\xi_0 \ne 0$ and using the same notation as before, we have
\begin{equation*}
\begin{split}
 ||\text{p.v.}K \ast f||_1 \leq
|\widehat{\text{p.v.}K}(\xi_0)| \, ||f||_1 + ||T_{E_2}f||_1, \\
||S\chi_B \ast f||_{1,\infty} \geq
\frac{1}{2}|\widehat{S\chi_B}(\xi_0)| \, \|f\|_{1,\infty} -||T_{E_1}f||_{1,\infty}
\end{split}
\end{equation*}
and consequently
\begin{equation*} 
|\widehat{S\chi_B}(\xi_0)|\leq C \left(
|\widehat{\text{p.v.}K}(\xi_0)| \frac{\|f\|_1}{\|f\|_{1,\infty}}  +
\frac{\|T_{E_2}f\|_1}{\|f\|_{1,\infty}}+\frac{\|T_{E_1}f\|_{1,\infty}}{\|f\|_{1,\infty}}
\right),
\quad  \xi_0 \neq 0.
\end{equation*}
Replacing $f$ by $g_{\delta}$ and using the properties of $g_{\delta}$ (that is, $\| g_{\delta}\|_1 =
\| g\|_1$,  $\| g_{\delta}\|_{1,\infty} = \| g\|_{1,\infty}$  and Lemma \ref{aux}) we obtain
\begin{align*}
|\widehat{S\chi_B}(\xi_0)|& \leq C \left(
|\widehat{\text{p.v.}K}(\xi_0)| \frac{\|g_{\delta}\|_1}{\|g_{\delta}\|_{1,\infty}}  
+\frac{\|T_{E_2 \phi_{2\delta}}g_{\delta}\|_1}{\|g_{\delta}\|_{1,\infty}}
+\frac{\|T_{E_1\phi_{2\delta}}g_{\delta}\|_{1,\infty}}{\|g_{\delta}\|_{1,\infty}}
\right) \\[2mm]
&\leq C \left(
|\widehat{\text{p.v.}K}(\xi_0)| \frac{\|g\|_1}{\|g\|_{1,\infty}}  +\frac{\|T_{E_2 \phi_{2\delta}}\|_{H^1\rightarrow
L^1}\|g_{\delta}\|_{H^1}}{\|g_{\delta}\|_{1,\infty}}
+\frac{\|T_{E_1\phi_{2\delta}}\|_{L^1\rightarrow L^{1,\infty}}\|g_{\delta}\|_{1}}{\|g_{\delta}\|_{1,\infty}}
\right)\\[2mm]
&\leq C \left(
|\widehat{\text{p.v.}K}(\xi_0)| + \|T_{E_2 \phi_{2\delta}}\|_{H^1\rightarrow L^1} + \|T_{E_1\phi_{2\delta}}\|_{L^1\rightarrow L^{1,\infty}}
\right)
\end{align*}
and therefore, applying Lemma  \ref{L1} on the right hand side of this inequality,
 we get
$$
|\widehat{S\chi_B}(\xi_0)| \leq C |\widehat{\text{p.v.}K}(\xi_0)|\quad  \xi_0 \neq 0
$$
as desired.

\section{The general case}\label{sec4}

In our procedure for the polynomial case, the function $b$  has been crucial. It provides a convenient way to express 
the function $K(x)\chi_{\Rn\setminus\overline{B}}$, where $K$ is the kernel of the operator $T$.
As we mentioned before, $b$ only depends on the degree of the homogeneous polynomial and on the space $\mathbb R^n$. In the even 
case $2N$ (see \eqref{b}), $b=b_{2N}$ is the restriction to the unit ball of some polynomial of degree $2N-2$. In the odd case $2N+1$, 
$b_{2N+1}$ is a BMO function with certain decay at infinity. 
Until now, we did not need to pay attention to the size of the parameters appearing in the definition 
of $b$ because the degree of the polynomial (either $2N$ or $2N +1$) was fixed. In this section we require a control of the $L^1$, 
$L^{\infty}$ or BMO norms of $b$, as well as its decay at infinity. We summarize all we need in next lemma.

\begin{lemma}\label{L4}
 There exists a constant $C$ depending only on $n$ such that
 \begin{enumerate}[(i)]
 \item $|\widehat{b_{2N}}(\xi)|\le C \quad$  and  $\quad |\widehat{b_{2N+1}}(\xi)|\le C$,  $\quad \xi\in \mathbb R^n$.
 \item $\| b_{2N}\|_{L^{\infty}(B)}\le C (2N)^{2n+2}\quad$ and $\quad \| \nabla b_{2N}\|_{L^{\infty}(B)}\le C(2N)^{2n+4}$.
 \item $\| b_{2N+1}\|_{\operatorname{BMO}}\le C (2N+1)^{2n}\quad$ and $\quad \| b_{2N+1}\|_{L^{2}}\le C(2N+1)^{2n}$.
 \item If $|x|>2\quad$ then $\quad |b_{2N+1}(x)|\le C(2N+1)^{2n}|x|^{-n-1}$.
 \end{enumerate}
\end{lemma}
\begin{proof}
 Parts \textit{(i)}, \textit{(ii)} and \textit{(iii)} are proved in \cite[Lemma 8]{MOV} and \cite[Lemma 5]{MOPV}.
 It only remains to prove \textit{(iv)}.
 
Recall that $\sigma$ denotes the normalized surface measure in $S^{n-1}$, and let $h_1,\dotsc,h_d$ be an orthonormal basis of
the subspace of $L^2(d\sigma)$ consisting of all homogeneous
harmonic polynomials of degree $2N+1$. As it is well known, $d \simeq (2N+1)^{n-2}$. As in the proof of Lemma 6 in
\cite{MOV} we have $h_1^2 +\dotsb+h_d^2 = d$, on $S^{n-1}$. Set
$$
H_j(x)= \frac{1}{\gamma_{2N+1} \sqrt{d}}\,h_j(x), \quad x \in \Rn\,,
$$
and let $S_j$ be the higher order Riesz transform with kernel
$K_j(x) = H_j(x)/|x|^{2N+1+n}$. The Fourier multiplier of $S_j^2$ is
$$
\frac{1}{d}\,\frac{h_j(\xi)^2}{|\xi|^{4N+2}}, \quad 0 \neq \xi \in
\Rn\,,
$$
and thus
\begin{equation}\label{16}
 \sum_{j=1}^d S_j^2 = \operatorname{Identity} \,.
\end{equation}

We use again \eqref{expressio}, but now the second term at the right hand side vanishes because 
each $h_{j}$ is harmonic (see \cite{MOV}, p. 1437).
We get
$$
K_j(x)\,\chi_{\BC}(x)=  S_j(b_{2N+1})(x),\quad x \in \Rn,\quad 1 \leq j
\leq d\,,
$$
and so by \eqref{16}
\begin{equation}\label{eq39}
b_{2N+1} = \sum_{j=1}^d S_j\left( K_j(x) \,\chi_{\BC}(x)\right)\,.
\end{equation}
Therefore we set
\begin{align*}
 \sum_{j=1}^d S_j\left( K_j(x) \,\chi_{\BC}(x)\right) & = \sum_{j=1}^d S_j*S_j -\sum_{j=1}^d S_j\left( K_j(x) \,\chi_{B}(x)\right)\\
 & = \delta_0  - \sum_{j=1}^d S_j\left( K_j(x) \,\chi_{B}(x)\right) ,
\end{align*}
where $\delta_0$ is the Dirac delta at the origin. If $|x|>2$, then
\begin{align*}
S_j ( K_j(y) \,\chi_{B}(y))(x) & = \lim_{\varepsilon\to 0} \int_{\varepsilon<|y|<1} K_j(x-y) K_j(y) dy\\
     & = \lim_{\varepsilon\to 0} \int_{\varepsilon<|y|<1} (K_j(x-y)-K_j(x)) K_j(y) dy .
\end{align*}
In this situation,
$$ 
|K_j(x-y)-K_j(x)|\le C\frac{|y|}{|x|^{n+1}} \left(\|H_j\|_\infty (2N+1)+\|\nabla H_j\|_\infty\right),
$$
hence
$$
\left| S_j ( K_j(y) \,\chi_{B}(y))(x)  \right|\le C\frac{\|H_j\|_\infty (2N+1)+\|\nabla H_j\|_\infty}{|x|^{n+1}}
\int_{|y|<1} \frac{\| H_j\|_\infty}{|y|^{n-1}} dy
$$
where the supremum norms are taken on $S^{n-1}$. Clearly
$$
\|H_j\|_\infty  = \frac{1}{\gamma_{2N+1}}
\left\|\frac{h_j}{\sqrt{d}}\right\|_\infty \leq \frac{1}{\gamma_{2N+1}} \simeq
(2N+1)^{n/2}\,.
$$
For the estimate of the gradient of $H_j$ we use the
inequality~\cite[p.~276]{St}
\begin{equation*}\label{eq40bis}
\| \nabla H_j\|_\infty \leq C\,(2N+1)^{n/2+1}\,\|H_j\|_2\,,
\end{equation*}
where the $L^2$ norm is taken with respect to $d\sigma$. Since the
$h_j$ are an orthonormal system,
$$
\|H_j\|_2 = \frac{1}{\sqrt{d}\,\gamma_{2N+1}} \simeq
\frac{(2N+1)^{n/2}}{(2N+1)^{(n-2)/2}} \simeq 2N+1\,.
$$
Gathering the above inequalities we get, when $|x|>2$,
$$
\left| S_j ( K_j(y) \,\chi_{B}(y))(x)  \right|\le C\frac{(2N+1)^{n+2}}{|x|^{n+1}}
$$
and finally
$$
|b_{2N+1}(x)|\le C d \frac{(2N+1)^{n+2}}{|x|^{n+1}}\le C\frac{(2N+1)^{2n}}{|x|^{n+1}},
$$
as claimed. 

\end{proof}

Now, the kernel of the operator $Tf=\text{p.v.}K\ast f$ is of the type
$K(x)=\frac{\Omega(x)}{|x|^n}$ being $\Omega$ a $\emph{C}^{\infty}(S^{n-1})$ homogeneous function
of degree $0$, with vanishing integral on the sphere. Then, $\Omega (x)=\sum_{j\geq
1}^{\infty}\frac{P_{2j}(x)}{|x|^{2j}}$ with $P_{2j}$ homogeneous
harmonic polynomials of degree $2j$ when $T$ is an even operator, and 
$\Omega (x)=\sum_{j\geq
0}^{\infty}\frac{P_{2j+1}(x)}{|x|^{2j+1}}$ with $P_{2j+1}$ homogeneous
harmonic polynomials of degree $2j+1$ when $T$ is an odd operator. The strategy consists in
passing to the polynomial case by looking at a partial sum of the
series above. Set, for each $N\geq 1$, $K_N(x)=\frac{\Omega_N
(x)}{|x|^n}$, where
$\Omega_N(x)=\sum_{j=1}^{N}\frac{P_{2j}(x)}{|x|^{2j}}$ (or $\Omega_N(x)=\sum_{j=0}^{N}\frac{P_{2j+1}(x)}{|x|^{2j+1}}$ in the odd case), 
and let $T_N$ be the operator with kernel $K_N$. 

We begin by considering \textit{(b)} implies \textit{(c)} in Theorem \ref{teo1} when $\omega=1$, 
that is, $T$ is even and our hypothesis is
$\|T^{*} f\|_p \leq C\|Tf\|_p$, $f\in L^p(\mathbb R^n)$.
In this setting, the difficulty is
that there is no obvious way of obtaining the inequality
\begin{equation}\label{controlTN}
\|T_{N}^{*}f\|_p\leq C\|T_N f\|_p, \quad f\in L^p(\mathbb R^n).
\end{equation}
Instead, we try to get \eqref{controlTN} with $\|T_N f\|_p$ replaced
by $ \|Tf\|_p$ in the right hand side plus an additional term
which becomes small as $N$ tends to $\infty$. We start by writing
\begin{equation}\label{EqA}
 \begin{split}
 \|T_N^1f\|_p &\leq \|T^1f\|_p+\|T^1f-T_N^1f\|_p\\
 &\leq  C\|Tf\|_p+\|\sum_{j>N}\frac{P_{2j}(x)}{|x|^{2j+n}}\chi_{\overline{B}^c}\ast f\|_p.
\end{split}
\end{equation}

By \eqref{expressio}, and since every $P_{2j}$ is harmonic, there exists a bounded function $b_{2j}$ supported on $B$ such that

$$\frac{P_{2j}(x)}{|x|^{2j+n}}\chi_{\overline{B}^c}(x)= \text{p.v.} \frac{P_{2j}(x)}{|x|^{2j+n}}\ast b_{2j}.$$

By Lemma \ref{L4} \textit{(ii)} , we have that
$\|b_{2j}\|_{L^1}\le C \|b_{2j}\|_{L^{\infty}(B)}\leq C(2j)^{2n+2}$, and thus
\begin{equation}\label{EqB}
\begin{split}
\|\sum_{j>N}\frac{P_{2j}(x)}{|x|^{2j+n}}\chi_{\overline{B}^c}\ast
f\|_p &= \|\sum_{j>N} \text{p.v.} \frac{P_{2j}(x)}{|x|^{2j+n}}\ast b_{2j}\ast
f\|_p\\&\leq \sum_{j>N}\|\frac{P_{2j}(x)}{|x|^{2j+n}}\|_{L^p\rightarrow L^p}\|b_{2j}\ast f\|_p\\
&\leq \sum_{j>N}\|\frac{P_{2j}(x)}{|x|^{2j+n}}\|_{L^p\rightarrow L^p}\|b_{2j}\|_{1} \|f\|_p\\
&\leq C \|f\|_p \sum_{j>N}\|\frac{P_{2j}(x)}{|x|^{2j+n}}\|_{L^p\rightarrow L^p} (2j)^{2n+2}\\ 
&\leq C \|f\|_p \sum_{j>N}(\|P_{2j}\|_{\infty}+\|\nabla
P_{2j}\|_{\infty})(2j)^{2n+2}.
\end{split}
\end{equation}
The last inequality follows from a well-known estimate for Calder\'{o}n-Zygmund operators (e.g. \cite[Theorem 4.3.3]{Gr1}).
On the other hand, 
$$
K_N(x)\chi_{\Rn\setminus\overline{B}}(x)=T_N(b_{2N})(x)+S_N(x)\chi_B(x)
$$
and then
$$
T_N^1f=\text{p.v.}K_N\ast b_{2N} \ast f + S_N\chi_B \ast f.
$$

So, for each $f\in L^p(\Rn)$, using \eqref{EqA} and \eqref{EqB}, we have the $L^p$ inequality
\begin{align*}
&\|S_N\chi_B \ast f\|_p \leq \|T_N^1f\|_p+\|\text{p.v.}K_N\ast b_{2N} \ast f\|_p\\
&\leq C\left(\|Tf\|_p+ \|f\|_p \sum_{j>N}(\|P_{2j}\|_{\infty}+\|\nabla
P_{2j}\|_{\infty})(2j)^{2n+2}  +\|\text{p.v.}K_N\ast b_{2N} \ast
f\|_p\right).
\end{align*}
We emphasize that the corresponding multipliers $\widehat{S_N\chi_B}$, $\widehat{\text{p.v.}K}$ and $\widehat{\text{p.v.}K_N\ast b_{2N}}= 
\widehat{\text{p.v.}K_N} \widehat{b_{2N}}$ are in $\mathcal{C}^{\infty}\setminus \{0\}$ and in $\mathcal{M}_p$.
Therefore, proceeding as in the polynomial case, and applying Lemma \ref{L1} we obtain the pointwise estimate for $\xi\neq 0$
\begin{align*}\label{generalcase}
|\widehat{S_N\chi_B}(\xi)| &\leq C\left(
|\widehat{\text{p.v.}K}(\xi)|+|(\widehat{\text{p.v.}K_N}\cdot
\widehat{b_{2N}})(\xi)|+\sum_{j>N}(\|P_{2j}\|_{\infty}+\|\nabla
P_{2j}\|_{\infty})(2j)^{2n+2}\right)\\ &\leq
C\left(|\widehat{\text{p.v.}K}(\xi)| +
|\widehat{\text{p.v.}K_N}(\xi)|+\sum_{j>N}(\|P_{2j}\|_{\infty}+\|\nabla
P_{2j}\|_{\infty})(2j)^{2n+2}\right),
\end{align*}
where in the last step we have used Lemma \ref{L4} \textit{(i)},
that is, $|\widehat{b_{2N}}(\xi)|\leq C$, for $\xi\in\Rn$.\\

The idea is now to take limits, as $N$ goes to $\infty$, in the
preceding inequality. By the definition of $K_N$ and \eqref{infinit}, the term on the right-hand side converges 
to $C |\widehat{\text{p.v.}K}(\xi)|$.
The next task is to clarify how the left-hand side converges, 
but at this point we proceed as in \cite[p. 1463]{MOV} and we get the desired result.\\

This argument, which has been explained for the even case and $\omega =1$, is also valid for the other cases, 
after taking into account the particular details listed below.\\

To get \textit{(b)} implies \textit{(c)}  in Theorem \ref{teo1} for any $\omega\in A_p$, we would use
$$ 
\| b_{2j}*f\|_{L^p(\omega)}\le C \|b_{2j}\|_{L^{\infty}(B)} \| Mf\|_{L^p(\omega)} \le
C (2j)^{2n+2}\| f\|_{L^p(\omega)}
$$
to obtain the inequality analogous to \eqref{EqB}

In order to obtain \textit{(d)} implies \textit{(c)} in Theorem \ref{teo1}, note that if $c_j> 0$ and $\sum_{j=1}^{\infty} c_j= 1$, 
then $\|\sum g_j\|_{1,\infty}
\le \sum c_j^{-1} \|g_j\|_{1,\infty}$. We have
\begin{align*}
\|\sum_{j>N}\frac{P_{2j}(x)}{|x|^{2j+n}}\chi_{\overline{B}^c}\ast
f\|_{1,\infty} &= \|\sum_{j>N} \text{p.v.} \frac{P_{2j}(x)}{|x|^{2j+n}}\ast b_{2j}\ast
f\|_{1,\infty}\\
&\le \sum_{j>N}j^2\| \text{p.v.} \frac{P_{2j}(x)}{|x|^{2j+n}}\ast b_{2j}\ast f\|_{1,\infty}\\
&\leq \sum_{j>N}j^2\|\frac{P_{2j}(x)}{|x|^{2j+n}}\|_{L^1\rightarrow L^{1,\infty}}\|b_{2j}\ast f\|_1\\
&\leq \sum_{j>N}j^2\|\frac{P_{2j}(x)}{|x|^{2j+n}}\|_{L^1\rightarrow L^{1,\infty}}\|b_{2j}\|_{1} \|f\|_1\\
&\leq C\|f\|_1\sum_{j>N}j^2\|\frac{P_{2j}(x)}{|x|^{2j+n}}\|_{L^1\rightarrow L^{1,\infty}} (2j)^{2n+2}\\ 
&\leq C \|f\|_1 \sum_{j>N}(\|P_{2j}\|_{\infty}+\|\nabla P_{2j}\|_{\infty})(2j)^{2n+4},
\end{align*}
and therefore, for all functions $f\in H^1(\mathbb R^n$),
\begin{align*}
\|S_N\chi_B \ast f\|_{1,\infty}&\leq 2(\|T_N^1f\|_{1,\infty}+\|\text{p.v.}K_N\ast b_{2N} \ast f\|_{1,\infty})\\
& \le 4(\|T^1 f\|_{1,\infty} + \|\sum_{j>N}\frac{P_{2j}(x)}{|x|^{2j+n}}\chi_{\overline{B}^c}\ast
f\|_{1,\infty} ) + 2\|\text{p.v.}K_N\ast b_{2N} \ast f\|_{1,\infty})\\
&\leq C (\|Tf\|_1+ \|f\|_1 \sum_{j>N}(\|P_{2j}\|_{\infty}+\|\nabla
P_{2j}\|_{\infty})(2j)^{2n+4} +\\
& \qquad \quad  +\quad \|\text{p.v.}K_N\ast b_{2N} \ast f\|_{1, \infty} ).
\end{align*}
Again, using Lemma \ref{L1}, Lemma \ref{aux} and Lemma \ref{L4}, we obtain, for $\xi\neq 0$,
$$
|\widehat{S_N\chi_B}(\xi)| \leq
C\left(|\widehat{\text{p.v.}K}(\xi)| +
|\widehat{\text{p.v.}K_N}(\xi)|+\sum_{j>N}(\|P_{2j}\|_{\infty}+\|\nabla
P_{2j}\|_{\infty})(2j)^{2n+4}\right)
$$
 as desired.

The implication $(b)\Rightarrow  (c)$ in Theorem \ref{teo2} can be adapted as follows. $T$ is odd and the functions $b_{2j+1}$ 
are in $\operatorname{BMO}$. By Lemma \ref{L4}, we have
$\|\widehat{b_{2j+1}}\|_{\infty}\le C$, $\|b_{2j+1}\|_{\operatorname{BMO}}\le C(2j+1)^{2n}$ and $\|b_{2j+1}\|_{2}\le C(2j+1)^{2n}$.
Moreover, $|b_{2j+1}(x)| \le C(2j+1)^{2n} |x|^{-n-1}$ if $|x| > 2$. Then, proceeding in the same way as in the proof of \eqref{iterada},
we get
$$ 
\| b_{2j+1}*f\|_{L^p(\omega)}\le C (2j+1)^{2n}\| f\|_{L^p(\omega)}
$$
and so, the inequality analogous to \eqref{EqB} follows.

\begin{gracies}
Thanks are due to Javier Duoandikoetxea for useful conversations on $L^p$ multipliers.
The authors were partially supported by grants\newline 2009SGR420
(Generalitat de Catalunya) and  MTM2010-15657 (Ministerio de Ciencia
e Innovaci\'{o}n, Spain).

\end{gracies}

\vspace{0.5 truecm}

\begin{tabular}{l}

Departament de Matem\`{a}tiques\\
Universitat Aut\`{o}noma de Barcelona\\
08193 Bellaterra, Barcelona, Catalonia\\\\

{\it E-mail:} {\tt abosch@mat.uab.cat}\\ 
{\it E-mail:} {\tt mateu@mat.uab.cat}\\ 
{\it E-mail:} {\tt orobitg@mat.uab.cat}\\
\end{tabular}


\vspace{0.5 truecm}

\begin{thebibliography}{}

\bibitem[BeL]{BeL}  J.\ Bergh and J.\ L\"ofstr\"om, {\em Interpolation spaces. An introduction}. 
Grundlehren der Mathematischen Wissenschaften {\bf 223}. Springer-Verlag, Berlin-New York (1976). 

\bibitem[Du]{Du} J.\ Duoandikoetxea, {\em Fourier Analysis}, Graduate Studies in Mathematics {\bf 29}, 
American Mathematical Society, Providence, RI (2001). 

\bibitem[Gr1]{Gr1} L. Grafakos, {\em Classical Fourier Analysis}, Graduate Texts in Mathematics {\bf 249},
Springer Verlag, Berlin, Second Edition (2008). 

\bibitem[Gr2]{Gr2} L. Grafakos, {\em Modern Fourier Analysis}, Graduate Texts in Mathematics {\bf 250},
Springer Verlag, Berlin, Second Edition (2008). 

\bibitem[KW]{KW} D. S. Kurtz, R. L. Wheeden, \emph{Results on Weighted norm inequalities for
multipliers}, Transactions of the American Mathematical Society
{\bf 255} (1979), 343--362.

\bibitem[MV]{MV} J.\ Mateu and J.\ Verdera, 
{\em $L^p$ and weak $L^1$ estimates for the maximal Riesz transform and the maximal Beurling transform} 
Math. Res. Lett. {\bf 13} (2006), 957--966. 
 
\bibitem[MOV]{MOV} J.\ Mateu, J.\ Orobitg and J.\ Verdera,
{\em Estimates for the maximal singular integral in terms of the
singular integral: the case of even kernels} Ann. of Math. {\bf 174} (2011), 1429--1483.

\bibitem[MOPV]{MOPV} J.\ Mateu, J.\ Orobitg, C.\ Perez, and J.\ Verdera,
{\em New estimates for the maximal singular integral} Int. Math. Res. Not. {\bf 19} (2010), 3658--3722.

\bibitem[P]{P}
C. P\'erez, {\em Weighted norm inequalities for singular integral
operators}, J. London Math. Soc., {\bf 49} (1994), 296--308.

\bibitem [St]{St} E.\ M.\ Stein,  {\em Singular integrals and differentiability properties of
functions,} Princeton University Press, Princeton, (1970).

\end{thebibliography}
\end{document}